\documentclass{amsart}
\usepackage{graphicx}
\usepackage{amssymb}
\usepackage{epstopdf}
\usepackage{epsfig}
\usepackage{subfigure}
\newtheorem{theorem}{Theorem}[section]

\newtheorem{proposition}{Proposition}[section]

\theoremstyle{remark}
\newtheorem{remark}{Remark}[section]
\theoremstyle{remark}

\newtheorem{definition}{Definition}[section]
\newtheorem{example}{Example}[section]
\numberwithin{equation}{section}

\theoremstyle{definition}
\theoremstyle{remark}
\numberwithin{equation}{section}

\begin{document}

\title[On Spherical Slant Helices in Euclidean 3-space]{%
On Spherical Slant Helices in Euclidean 3-space}
\author{\c{C}etin Camc{\i}, Levent Kula, Mesut Alt{\i}nok}
\address[\c{C}etin Camc{\i}]{\c{C}anakkale Onsekiz Mart University, Faculty of Sciences and Arts\\
Department of Mathematics, Canakkale, Turkey}
\email{ccamci@comu.edu.tr}
\address[Levent Kula]{Ahi Evran University, Faculty of Sciences and Arts\\
Department of Mathematics, Kirsehir, Turkey}
\email{lkula@ahievran.edu.tr}
\address[Mesut Alt{\i}nok]{Ahi Evran University, Faculty of Sciences and Arts\\
Department of Mathematics, Kirsehir, Turkey}
\email{altnokmesut@gmail.com}
\subjclass[2000]{53A04, 14H50.}

\begin{abstract}
In this paper we consider the spherical slant helices in $\mathbb{R}^{3}$. Moreover, we show how could be obtained to a spherical slant helix and we give some spherical slant helix examples in Euclidean 3-space.
~\\

\noindent\textbf{Key Words:} Slant helix, spherical curve, geodesic curvature.
\end{abstract}

\maketitle

\section{Introduction}

In \cite{izu}, A slant helix in Euclidean space $\mathbb{R}^{3}$
was defined by the property that the principal normal makes a constant angle
with a fixed direction. Moreover, Izumiya and Takeuchi showed that $\alpha $
is a slant helix in $\mathbb{R}^{3}$ if and only if the geodesic curvature
of the principal normal of a space curve $\alpha $ is a constant function.

In \cite{kula2}, Kula and Yayli have studied spherical images of
tangent indicatrix and binormal indicatrix of a slant helix and they showed
that the spherical images are spherical helix.

In \cite{kula1}, Kula, Ekmekci, Yayli and Ilarslan have studied
the relationship between the plane curves and slant helices in
$\mathbb{R}^{3}$. They obtained that the differential equations which are
characterizations of a slant helix.

In this paper we consider the spherical slant helices in $\mathbb{R}^{3}$. We also present the parametric slant helices, their curvatures and torsions. Moreover, we give some slant helix examples in Euclidean 3-space.

\section{Preliminaries}

We now recall some basic concept on classical geometry of space curves and the definition of slant helix in $\mathbb{R}^3$. A curve $\alpha:I \subset \mathbb{R} \to \mathbb{R}^3$ is a space curve. $T(s)=\frac{\alpha'(s)}{\|\alpha'(s)\|}$ is a unit tangent vector of $\alpha$ at $s$. If $\kappa(s)\neq0$, then the unit principal binormal vector $B(s)$ of the curve $\alpha$ at $s$ is given by $\frac{\alpha'(s) \wedge \alpha''(s)}{\|\alpha'(s) \wedge \alpha''(s)\|}$. The unit vector $N(s)=B(s) \wedge T(s)$ is called the unit normal vector of $\alpha$ at $s$. For the derivatives of the Frenet frame the Serret-Frenet formula hold:
\begin{eqnarray}
T'(s)&=&\nu \kappa(s)N(s), \nonumber\\
N'(s)&=&\nu(-\kappa(s)T(s)+\tau(s)B(s)),\\
B'(s)&=&-\nu \tau(s)N(s) \nonumber,
\label{1.1}
\end{eqnarray}
where $\kappa$ is the curvature of the curve $\alpha$ at $s$, $\tau$ is the torsion of the curve $\alpha$ at $s$ and $\nu = \|\alpha'\|$. For a general parameter $s$ of a space curve $\alpha$, we can calculate $\kappa$ the curvature and the torsion $\tau$ as follows:
\begin{equation*}
\kappa(s)=\frac{\|\alpha'(s) \wedge \alpha''(s)\|}{\|\alpha'(s)\|^3},~~~~\tau(s)=\frac{\det{\alpha'(s),\alpha''(s),\alpha'''(s)}}{\|\alpha'(s) \wedge \alpha''(s)\|^2}.
\end{equation*}
\begin{definition}
A curve $\alpha$ with $\kappa(s)\neq0$ is called a slant helix if the principal normal vector line of $\alpha$ make a constant angle with a fixed direction \cite{izu}.
\end{definition}
\begin{theorem}
$\alpha$ is a slant helix if and only if the geodesic curvature of the spherical image of the principal normal indicatrix (N) of $\alpha$
\begin{equation}
\sigma(s)=\left(\frac{\kappa^2}{\nu (\kappa^2+\tau^2)^{\frac{3}{2}}}\left(\frac{\tau}{\kappa}\right)'\right)(s)
\label{1.2}
\end{equation}
is a constant function \cite{izu}.
\end{theorem}

\begin{definition}
Specially, if $\kappa(s)=1$, a slant helix is called a Salkowski curve and if $\tau(s)=1$, a slant helix is called an anti-Salkowski curve \cite{monterde}.
\end{definition}

\begin{theorem}
\label{lemma2}
Let $\alpha$ be a space curve in $\mathbb{R}^3$. The following statements are equivalent:
\begin{enumerate}
  \item $\alpha$ is a spherical, i.e., it is contained in a sphere of radius $r$
  \item $ \left(\frac{1}{\kappa}\right)^2 + \left(\frac{1}{\nu \tau} \left(\frac{1}{\kappa}\right)' \right)^2=r^2.$
  \item $\frac{1}{\nu}\left(\frac{1}{\nu \tau} \left(\frac{1}{\kappa}\right)' \right)' + \frac{\tau}{\kappa}=0.$
  \item $\frac{1}{\kappa} = A \cos\left(\int \nu \tau ds \right) + B \sin\left(\int \nu \tau ds \right)$\\~\\
where, $A$, $B$ are constant and $\sqrt{A^2+B^2}=r$ \cite{wong}.
\end{enumerate}
\end{theorem}

\section{Slant helix and Its Projection}

In this section, we investigate curvature of slant helix and projection of slant helix in the plane.
\begin{theorem}
Let $\alpha$ be a space curve in $\mathbb{R}^3$. The following statements are equivalent:
\begin{enumerate}
  \item $\alpha$ is a slant helix.
  \item
  \begin{eqnarray*}
\kappa & = & \frac{1}{a \nu} \theta' \sin \theta\\
\tau & = & \frac{1}{a \nu} \theta' \cos \theta.
\end{eqnarray*}
  \item $\vec{U}=\frac{ \tau }{a \sqrt{\kappa^2 + \tau^2}}  ~T+  ~N + \frac{ \kappa }{a \sqrt{\kappa^2 + \tau^2}} ~B=constant$.
\end{enumerate}
Where $a\neq 0$ is a constant and $\theta$ is a function of $s$ \cite{toni}.
\end{theorem}

\begin{theorem}
Let $\alpha$ be a slant helix with Frenet frame $\{T,N,B\}$, curvature $\kappa$, torsion $\tau$ and $\vec{a}=\cos\theta_1 T + \cos\theta_2 N + \cos\theta_3 B$ be axis of slant helix.
\begin{equation}
\label{11}
\alpha_{\pi}(s)=\alpha(s)- \langle \alpha(s), \vec{a} \rangle \vec{a}
\end{equation}
is a plane curve and its curvature $\kappa_{\pi}$ is
\begin{eqnarray}
\label{18}
\kappa_{\pi}=\frac{(1 +a^2) \sin \theta}{(a^2+\sin^2 \theta)^{\frac{3}{2}}} \kappa.
\end{eqnarray}
\end{theorem}

\begin{proof}
Differentiating the eq. \eqref{11}, we get
\begin{eqnarray}
\label{12}
\alpha_{\pi}'(s)=\nu(T - \cos \theta_1 \vec{a}).
\end{eqnarray}
Therefore
\begin{eqnarray*}
\label{14}
\nu_{\pi}&=& \| \alpha_{\pi}'(s) \| = \frac{ds_{\pi}}{ds} =\nu \| T - \cos \theta_1 \vec{a} \|=\nu \sin \theta_1.
\end{eqnarray*}
If we derive eq. \eqref{12} again, we obtain
\begin{eqnarray*}
\label{42}
\alpha_{\pi}''(s)=\nu' T + \nu^2  \kappa N - \nu' \cos \theta_1 \vec{a} - \nu^2 \kappa \cos \theta_2 \vec{a}.
\end{eqnarray*}
Thus
\begin{eqnarray*}
\alpha_{\pi}' \times \alpha_{\pi}''&=&\nu^3 \kappa ( B - \cos \theta_2 T \times \vec{a} - \cos \theta_1 \vec{a} \times N)\\
&=&\nu^3 \kappa \cos \theta_3 \vec{a}.
\end{eqnarray*}
It follows
\begin{eqnarray}
\label{17}
\kappa_{\pi}=\frac{\|\alpha_{\pi}' \times \alpha_{\pi}''\|}{\nu_{\pi}^3}=\frac{\cos \theta_3 }{\sin^3 \theta_1} \kappa
\end{eqnarray}
and $\tau_{\pi}=0$.
We can easily see that
\begin{eqnarray}
\cos \theta_1  = - \frac{\cos \theta}{\sqrt{1+a^2}}\nonumber, ~~~~~\cos \theta_2  =  \frac{a}{\sqrt{1+a^2}}, ~~~~~\cos \theta_3  =  \frac{\sin \theta}{\sqrt{1+a^2}} \nonumber
\end{eqnarray}
In this case,
\begin{eqnarray*}
\label{15}
\sin \theta_1 = \frac{\sqrt{a^2+\sin^2 \theta}}{\sqrt{1+a^2}},
\end{eqnarray*}
\begin{eqnarray}
\label{15}
\frac{ds_{\pi}}{ds}= \nu \frac{\sqrt{a^2+\sin^2 \theta}}{\sqrt{1+a^2}}
\end{eqnarray}
and finally,
\begin{eqnarray}
\label{25}
\kappa_{\pi}=\frac{(1 +a^2) \sin \theta}{(a^2+\sin^2 \theta)^{\frac{3}{2}}} \kappa.
\end{eqnarray}

\end{proof}

\section{Spherical slant helices}

In this section, we investigate parametric equation of spherical slant helix and we give some theorem.

The axis $\vec{a}$ of slant helix could chosen axis Oz without loss of generality. Now since the tangent and normal are orthogonal unit vectors, we may write $T_{\pi}(s_{\pi})=(\cos \phi_{\pi},\sin \phi_{\pi})$ and $N_{\pi}(s_{\pi})=(\sin \phi_{\pi},-\cos \phi_{\pi})$, $\phi_{\pi}(s_{\pi})$ being the angle between the $x$-axis and tangent. We observe that the curvature has the interpretation
\begin{equation*}
\kappa_{\pi}=\phi_{\pi}'
\end{equation*}
i.e., it is derivative of the tangent-angle $\phi_{\pi}$ with respect to the arc length $s_{\pi}$.

The function $\kappa_{\pi}(s_{\pi})$ specifying thr curvature in terms of arc length along a plane curve is called the intrinsic equation of that curve. it uniquely defines the curve. We have the explicit representation
\begin{equation}
\label{xyz}
x(s)  =  \int \cos \phi_{\pi} ds_{\pi}, ~~~~~~y(s) = \int \sin \phi_{\pi}  ds_{\pi}
\end{equation}
of the curve $\alpha(s)=(x(s),y(s),z(s))$, where
\begin{equation}
\label{xyzt}
\phi_{\pi}=\int \kappa_{\pi} ds_{\pi}.
\end{equation}
\cite{neill}.

Therefore, from eq. \eqref{xyzt} and eq. \eqref{25}
\begin{eqnarray}
\label{fi}
\phi_{\pi}& = & -\arctan \left( \frac{\sqrt{1+a^2}}{a} \tan \theta \right) + \frac{\sqrt{1+a^2}}{a} \theta.
\end{eqnarray}
and by using eq. \eqref{xyz}, eq. \eqref{15},
\begin{eqnarray*}
x(s)& = & \int \cos \phi_{\pi} ds_{\pi} \nonumber \\
&=& \int \left( \frac{a}{\sqrt{\sin^2\theta + a^2}} \cos\theta \cos \left[ \frac{\sqrt{1+a^2}}{a} \theta \right] + \frac{\sqrt{1+a^2}}{\sqrt{\sin^2\theta + a^2}} \sin\theta \sin\left[ \frac{\sqrt{1+a^2}}{a} \theta \right] \right) ds_{\pi} \nonumber \\
&=& \nu \int \left( \frac{a}{\sqrt{1+a^2}} \cos\theta \cos \left[ \frac{\sqrt{1+a^2}}{a} \theta \right] + \sin\theta \sin\left[ \frac{\sqrt{1+a^2}}{a} \theta \right] \right) ds \nonumber.
\end{eqnarray*}
Then from the fourth equation of Theorem 2.2 and the second equation of Theorem 3.1, we have
\begin{equation}
\label{ds}
ds=\frac{1}{a \nu}\left(A \sin\theta \cos\left[ \frac{\sin\theta}{a} \right] + B \sin\theta \sin\left[ \frac{\sin\theta}{a} \right]\right)d\theta.
\end{equation}
Therefore, by using eq. \eqref{ds},
\begin{eqnarray*}
\label{x(s)}
x(s) & =  &\frac{1}{\sqrt{1+a^2}}( \sqrt{1+a^2} \cos \theta \sin \left[ \frac{\sqrt{1+a^2}}{a} \theta \right] \left( B \cos \left[ \frac{\sin\theta}{a} \right] -A\sin \left[ \frac{\sin\theta}{a} \right] \right){} \nonumber\\
&& {} - \cos \left[ \frac{\sqrt{1+a^2}}{a} \theta \right] \left( \cos \left[ \frac{\sin\theta}{a} \right] (A + a B \sin \theta) + ( B-a A \sin \theta ) \sin \left[ \frac{\sin\theta}{a} \right] \right) ),\nonumber\\
\end{eqnarray*}

And similarly,
\begin{eqnarray*}
\label{y(s)}
y(s)  & =  & -\frac{1}{\sqrt{1+a^2}}( \sqrt{1+a^2} \cos \theta \cos \left[ \frac{\sqrt{1+a^2}}{a} \theta \right] \left( B \cos \left[ \frac{\sin\theta}{a} \right] -A\sin \left[ \frac{\sin\theta}{a} \right] \right) {}\nonumber\\
&& {} + \sin \left[ \frac{\sqrt{1+a^2}}{a} \theta \right] \left( \cos \left[ \frac{\sin\theta}{a} \right] (A + a B \sin \theta) + ( B-a A \sin \theta ) \sin \left[ \frac{\sin\theta}{a} \right] \right) ,\nonumber\\
\end{eqnarray*}
Moreover, since
\begin{eqnarray}
\langle\alpha''(s),\vec{a}\rangle=z''(s) =-\nu' \frac{\cos\theta}{\sqrt{1+a^2}}+ \nu^2 \frac{a}{\sqrt{1+a^2}} \kappa,
\end{eqnarray}
we find that
\begin{eqnarray}
z(s) = \frac{1}{\sqrt{1+a^2}} \cos \left[ \frac{\sin\theta}{a} \right](-a A + B \sin \theta)- \sin \left[ \frac{\sin\theta}{a} \right] (a B + A \sin \theta).\nonumber
\end{eqnarray}
where $A^2+B^2=1$, $a$ is constant and $\theta=\theta(s)$.

As a result of the above findings, we can give the following theorem.

\begin{theorem}
Under the above notation, $\alpha$ is a spherical slant helix. Moreover, all spherical slant helix can be constructed by above method.
\end{theorem}

\begin{proof}
Let $\alpha$ be a space curve with Frenet frame $\{T,N,B\}$, curvature $\kappa$ and torsion $\tau$.

In this case, we will show that $\sigma$ is constant and $\alpha \in S^2$.

By simple calculation, spherical indicatricies $T(s)$, $N(s)$, $B(s)$ of the curve $\alpha$, respectively, are
\begin{eqnarray*}
T(s)& =  &( \frac{a \cos \theta \cos \left[ \frac{\sqrt{1+a^2}}{a} \theta \right] }{ \sqrt{1+a^2}} + \sin \theta \sin \left[ \frac{\sqrt{1+a^2}}{a} \theta \right], \\
&& {} \frac{a \cos \theta \sin \left[ \frac{\sqrt{1+a^2}}{a} \theta \right] }{ \sqrt{1+a^2}} - \sin \theta \cos \left[ \frac{\sqrt{1+a^2}}{a} \theta \right], -\frac{\cos \theta}{\sqrt{1+a^2}}),\\
N(s)& =  &( \frac{ \cos \left[ \frac{\sqrt{1+a^2}}{a} \theta \right] }{ \sqrt{1+a^2}} , \frac{ \sin \left[ \frac{\sqrt{1+a^2}}{a} \theta \right] }{ \sqrt{1+a^2}} , \frac{a}{\sqrt{1+a^2}}),\\
B(s)& =  &( -\frac{a \sin \theta \cos \left[ \frac{\sqrt{1+a^2}}{a} \theta \right] }{ \sqrt{1+a^2}} + \cos \theta \sin \left[ \frac{\sqrt{1+a^2}}{a} \theta \right], \\
&& {} -\frac{a \sin \theta \sin \left[ \frac{\sqrt{1+a^2}}{a} \theta \right] }{ \sqrt{1+a^2}} - \cos \theta \cos \left[ \frac{\sqrt{1+a^2}}{a} \theta \right], \frac{\sin \theta}{\sqrt{1+a^2}}).
\end{eqnarray*}
And also, by the formulae of the curvature and the torsion for a general parameter, we can calculate that
\begin{eqnarray*}
\kappa(s) & = & \frac{1}{|A \cos \left[ \frac{\sin \theta}{a} \right]+B \sin \left[ \frac{\sin \theta}{a} \right]|}, \\
\tau(s) & = & \frac{\cot \theta}{A \cos \left[ \frac{\sin \theta}{a} \right]+B \sin \left[ \frac{\sin \theta}{a} \right]}.
\end{eqnarray*}

Therefore, by using eq. \eqref{1.2}
\begin{equation*}
\sigma(s)=-a=constant
\end{equation*}
which means that $\alpha$ is a slant helix.

Finally,
\begin{equation*}
\| \alpha \|=A^2+B^2=1.
\end{equation*}
Then, $\alpha$ is spherical slant helix. Thus the proof of theorem is completed.
\end{proof}

\begin{proposition}
If $\frac{\sqrt{1+a^2}}{a}$ is a rational number, then spherical slant helix $\alpha$ is closed curve. If $\frac{\sqrt{1+a^2}}{a}$ is not a rational number, then the curve spherical slant helix $\alpha$ never closes.
\end{proposition}

\begin{definition}
Let $\alpha$ be a spherical curve. Let us denote $T(s)=\nu(s) \alpha'(s)$, and we call $T$ a unit tangent vector of $\alpha$ at $s$. We now set a vector $Y(s)=\alpha(s) \wedge T(s)$. By definition, we have an orthonormal frame $\{\alpha(s),T(s),Y(s)\}$. This frame is called the Sabban frame of $\alpha$ \cite{koen}.
\end{definition}

\begin{theorem}
The indicatrix $Y$ of the spherical slant helix $\alpha$ is a spherical slant helix.
\end{theorem}
\begin{proof}
Let $Y$ be curve with Frenet frame $\{T^Y,N^Y,B^Y\}$, curvature $\kappa^Y$, torsion $\tau^Y$ and the geodesic curvature of the spherical image of the principal normal indicatrix $(N^Y)$ of $Y$ be $\sigma^Y$. Then we have
\begin{eqnarray*}
Y'(s) & = &  \nu \kappa \alpha \wedge N,\\
Y''(s) & = & (\nu \kappa)'  \alpha \wedge N + \nu^2 \kappa B -  \nu^2 \kappa^2 \alpha \wedge T +  \nu^2 \kappa \tau \alpha \wedge B,\\
Y''(s) & = & (\nu \kappa)'' \alpha \wedge N + (\nu \kappa)' \nu B + \nu \tau (\nu \kappa)' \alpha \wedge B - \nu \kappa (\nu \kappa)' \alpha \wedge T \\
&+& (\nu^2 \kappa)' B - \nu^3 \kappa \tau N - (\nu^2 \kappa^2)' \alpha \wedge T - \nu^3 \kappa^3 \alpha \wedge N\\
&+& (\nu^2 \kappa \tau)' \alpha \wedge B - \nu^3 \kappa \tau N - \nu^3 \kappa \tau^2 \alpha \wedge N.
\end{eqnarray*}
By the formulae the curvature and the torsion for a general parameter, we can calculate that
\begin{eqnarray*}
\kappa^Y & = & \frac{\kappa}{\sqrt{\kappa^2-1}}, \\
\tau^Y & = & \frac{\tau}{\sqrt{\kappa^2-1}}.
\end{eqnarray*}
Moreover, $\sigma^Y(s)=a$, so that $Y(s)$ is a spherical slant helix.
\end{proof}

\begin{remark}
There are not spherical Salkowski and anti-Salkowski curve.
\end{remark}

\begin{proof}
We choose that $\alpha$ is unit speed Salkowski curve without loss of generality. Since $\kappa(s)=1$ and by using
\begin{equation*}
\left(\frac{1}{\tau} \left(\frac{1}{\kappa}\right)' \right)' + \frac{\tau}{\kappa}=0,
\end{equation*}
we have $\tau(s)=0$. Therefore, $\alpha$ is plane curve and is not spherical slant helix.

Similarly, we choose that $\alpha$ is unit speed Anti-Salkowski curve without loss of generality. Then, by simple calculation, we have $\sigma(s)$ is not constant which means that $\alpha$ is not spherical slant curve.
\end{proof}

\section{Example}
In this section we give two example of spherical slant helices in Euclidean $3$%
-space and draw its pictures and its tangent indicatrix, normal indicatrix, and binormal
indicatrix by using \textbf{Mathematica}.

\begin{example}
For $a=1$, $A=1$, $B=0$, We consider a spherical slant helix $\alpha $ is defined by

\begin{eqnarray*}
x(s) & = & -\cos \theta \sin(\sqrt{2} \theta ) \sin(\sin \theta) {} \nonumber\\
&& {} +\frac{1}{\sqrt{2}} \cos(\sqrt{2}\theta)  (-\cos(\sin \theta)) + \sin \theta \sin(\sin \theta) ) \nonumber\\
y(s) & = & \cos \theta \cos(\sqrt{2}\theta) \sin(\sin \theta) {} \nonumber\\
&& {} +\frac{1}{\sqrt{2}} \sin(\sqrt{2}\theta)  (-\cos(\sin \theta)) + \sin \theta \sin(\sin \theta) ) \nonumber\\
z(s) &=& \frac{1}{\sqrt{2}} (-\cos(\sin \theta) + \sin \theta \sin(\sin \theta)). \nonumber
\end{eqnarray*}

The picture of the curve $\alpha $ is rendered in Figure 1.

The parametrization of the tangent indicatrix $T =\left( T_{1},T_{2},T_{3}\right) $ of the spherical slant helix $\alpha $ is

\begin{eqnarray*}
T_1(s) & = & \frac{1}{\sqrt{2}} \cos \theta \cos (\sqrt{2} \theta) + \sin(\sqrt{2} \theta) \sin \theta, \nonumber\\
T_2(s) & = & \frac{1}{\sqrt{2}} \sin \theta \cos (\sqrt{2} \theta) - \sin(\sqrt{2} \theta) \cos \theta, \nonumber\\
T_3(s) &=& -\frac{1}{\sqrt{2}} \cos \theta. \nonumber
\end{eqnarray*}

The picture of the tangent indicatrix is rendered in Figure 2 (a).

The parametrization of the normal indicatrix $N =\left( N_{1},N_{2},N_{3}\right) $ of the spherical slant helix $\alpha $
is

\begin{eqnarray*}
N_1(s) & = & \frac{1}{\sqrt{2}}  \cos (\sqrt{2} \theta), \nonumber\\
N_2(s) & = & \frac{1}{\sqrt{2}}  \sin (\sqrt{2} \theta) , \nonumber\\
N_3(s) &=& \frac{1}{\sqrt{2}} . \nonumber
\end{eqnarray*}

The picture of the normal indicatrix is rendered in Figure 2 (b).

The parametrization of the binormal indicatrix $B =\left( B_{1},B_{2},B_{3}\right) $ of the spherical slant helix $\alpha $
is

\begin{eqnarray*}
B_1(s) & = & -\frac{1}{\sqrt{2}} \sin \theta \cos (\sqrt{2} \theta) + \sin(\sqrt{2} \theta) \cos \theta, \nonumber\\
B_2(s) & = & -\frac{1}{\sqrt{2}} \sin \theta \sin (\sqrt{2} \theta) - \cos(\sqrt{2} \theta) \cos \theta, \nonumber\\
B_3(s) &=& \frac{1}{\sqrt{2}} \sin \theta. \nonumber
\end{eqnarray*}

The picture of the binormal indicatrix is rendered in Figure 2 (c).
\end{example}

\begin{example}
For $A=1$, $B=0$ and $a=2$, $a=3$, $a=4$, the picture of the spherical slant helix $\alpha $, respectively, is rendered in Figure 4.
\end{example}

\begin{example}
For $a=1$, $A=1$, $B=0$, We consider a spherical slant helix $Y$ is defined by

\begin{eqnarray*}
x(s) & = & -\cos \theta \sin(\sqrt{2} \theta) \cos(\sin \theta) {} \nonumber\\
&& {} +\frac{1}{\sqrt{2}} \cos(\sqrt{2}\theta)  (\cos(\sin \theta)) \sin \theta + \sin(\sin \theta) ) \nonumber\\
y(s) & = & \cos \theta \cos(\sqrt{2}\theta) \cos(\sin \theta) {} \nonumber\\
&& {} +\frac{1}{\sqrt{2}} \sin(\sqrt{2}\theta)  (\cos(\sin \theta)) \sin \theta + \sin(\sin \theta) ) \nonumber\\
z(s) &=& \frac{1}{\sqrt{2}} (-\cos(\sin \theta) \sin \theta +  \sin(\sin \theta)). \nonumber
\end{eqnarray*}

The picture of the curve $\alpha $ is rendered in Figure 4.

The parametrization of the tangent indicatrix $T^Y =\left( T^Y_1,T^Y_2,T^Y_3\right) $ of the spherical slant helix $Y$ is

\begin{eqnarray*}
T^Y_1(s) & = & -\frac{1}{\sqrt{2}} \cos \theta \cos (\sqrt{2} \theta) - \sin(\sqrt{2} \theta) \sin \theta, \nonumber\\
T^Y_2(s) & = & -\frac{1}{\sqrt{2}} \cos \theta \sin (\sqrt{2} \theta) + \cos(\sqrt{2} \theta) \sin \theta, \nonumber\\
T^Y_3(s) &=& \frac{1}{\sqrt{2}} \cos \theta. \nonumber
\end{eqnarray*}

The picture of the tangent indicatrix is rendered in Figure 5 (a).

The parametrization of the normal indicatrix $N^Y =\left( N^Y_1,N^Y_2,N^Y_3\right) $ of the spherical slant helix $Y$
is

\begin{eqnarray*}
N^Y_1(s) & = & -\frac{1}{\sqrt{2}}  \cos (\sqrt{2} \theta), \nonumber\\
N^Y_2(s) & = & -\frac{1}{\sqrt{2}}  \sin (\sqrt{2} \theta) , \nonumber\\
N^Y_3(s) &=& -\frac{1}{\sqrt{2}}.  \nonumber
\end{eqnarray*}

The picture of the normal indicatrix is rendered in Figure 5 (b).

The parametrization of the binormal indicatrix $B^Y=\left( B^Y_1,B^Y_2,B^Y_3\right) $ of the spherical slant helix $Y$ is

\begin{eqnarray*}
B^Y_1(s) & = & -\frac{1}{\sqrt{2}} \sin \theta \cos (\sqrt{2} \theta) + \sin(\sqrt{2} \theta) \cos \theta, \nonumber\\
B^Y_2(s) & = & -\frac{1}{\sqrt{2}} \sin \theta \sin (\sqrt{2} \theta) - \cos(\sqrt{2} \theta) \cos \theta, \nonumber\\
B^Y_3(s) &=& \frac{1}{\sqrt{2}} \sin \theta. \nonumber
\end{eqnarray*}

The picture of the binormal indicatrix is rendered in Figure 5 (c).
\end{example}

\begin{figure}[b]
\begin{center}
\includegraphics[width=2.in]{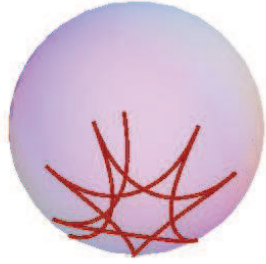}
\caption{For $a=1$, $A=1$, $B=0$, spherical slant helix $\alpha$.}
\label{1}
\end{center}
\end{figure}

\begin{figure}[b]
\centering
\subfigure[]{
   \includegraphics[scale =0.32] {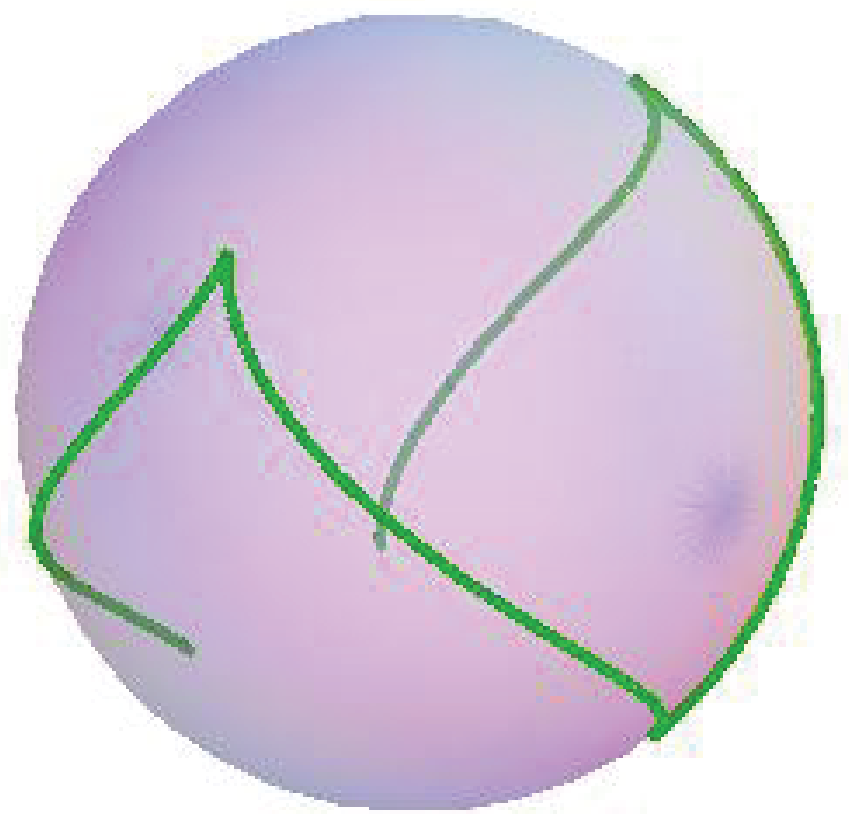}
 }
 \subfigure[]{
   \includegraphics[scale =0.32] {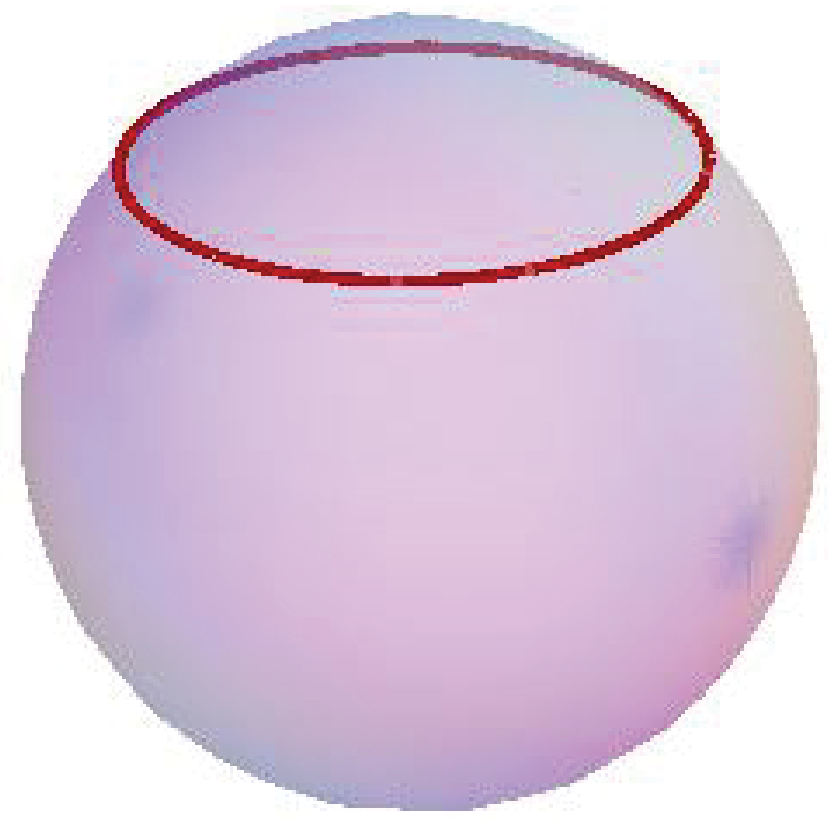}
 }
 \subfigure[]{
   \includegraphics[scale =0.32] {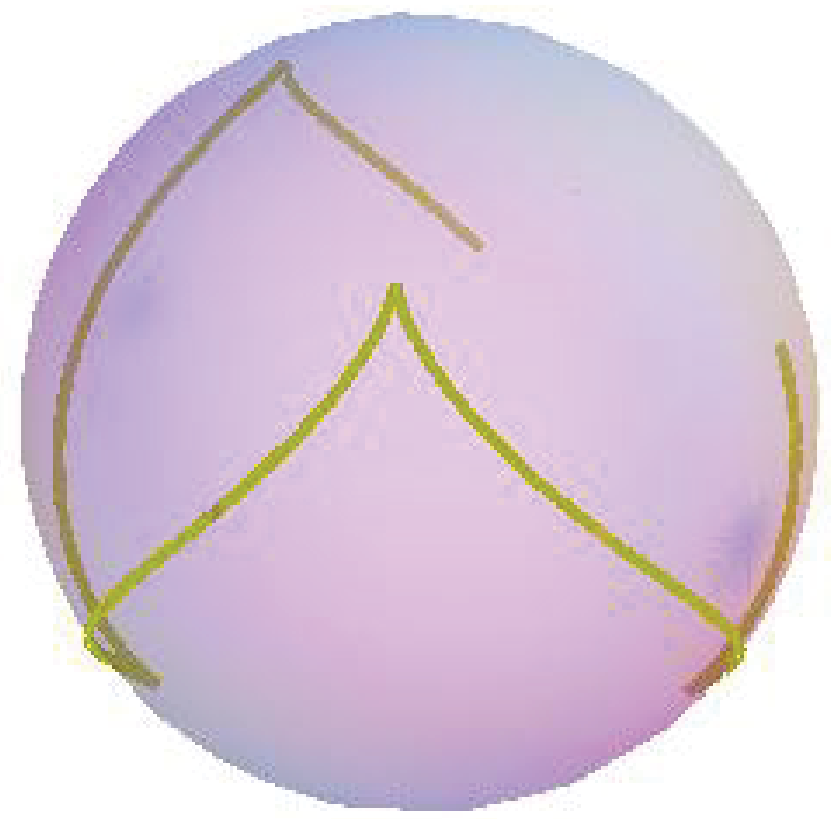}
 }
\label{cases1.5}
\caption{For $a=1$, $A=1$, $B=0$, tangent indicatrix of the spherical slant helix $\alpha$ (a), normal indicatrix of the spherical slant helix $\alpha$ (b) and binormal indicatrix of the spherical slant helix $\alpha$ (c).}
\end{figure}

\begin{figure}[h]
\centering
\subfigure[]{
   \includegraphics[scale =0.31] {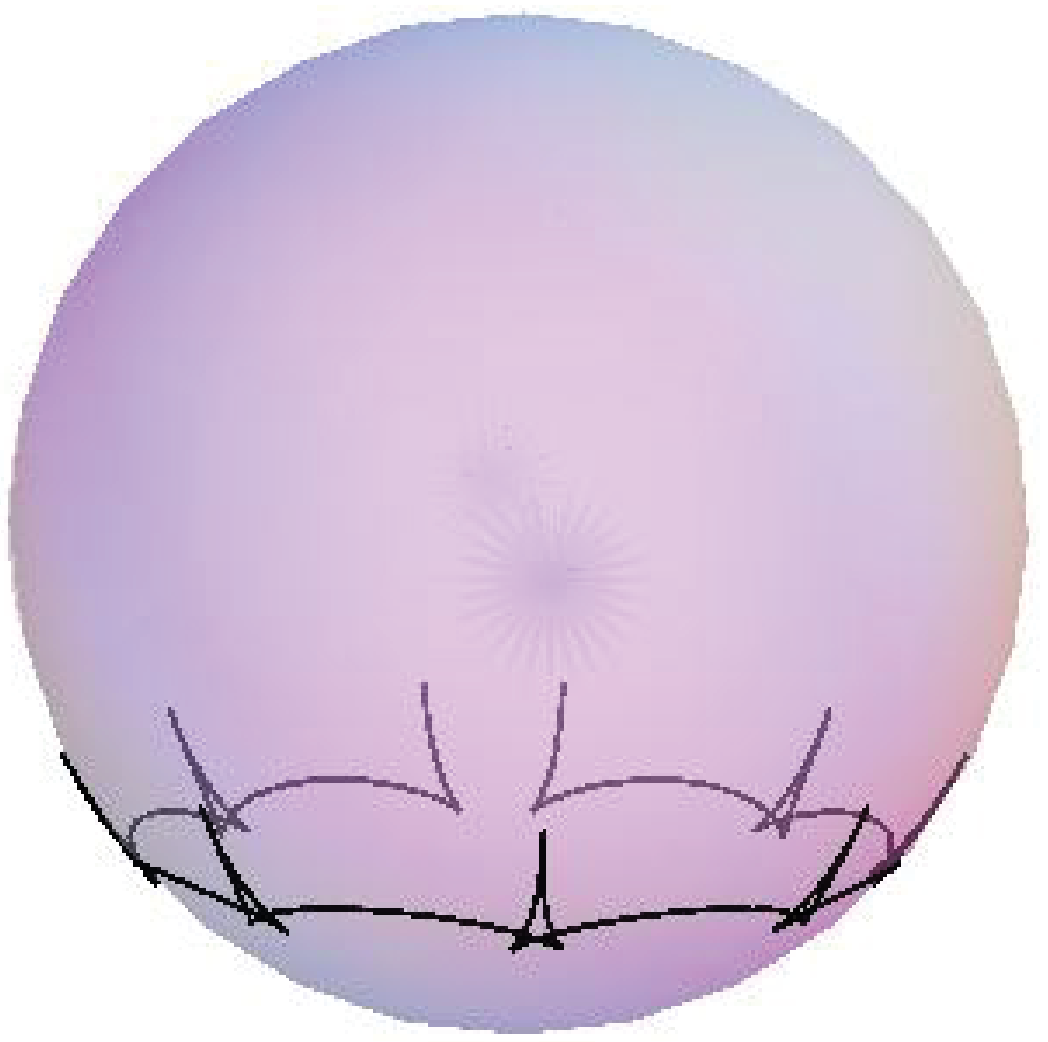}
 }
 \subfigure[]{
   \includegraphics[scale =0.31] {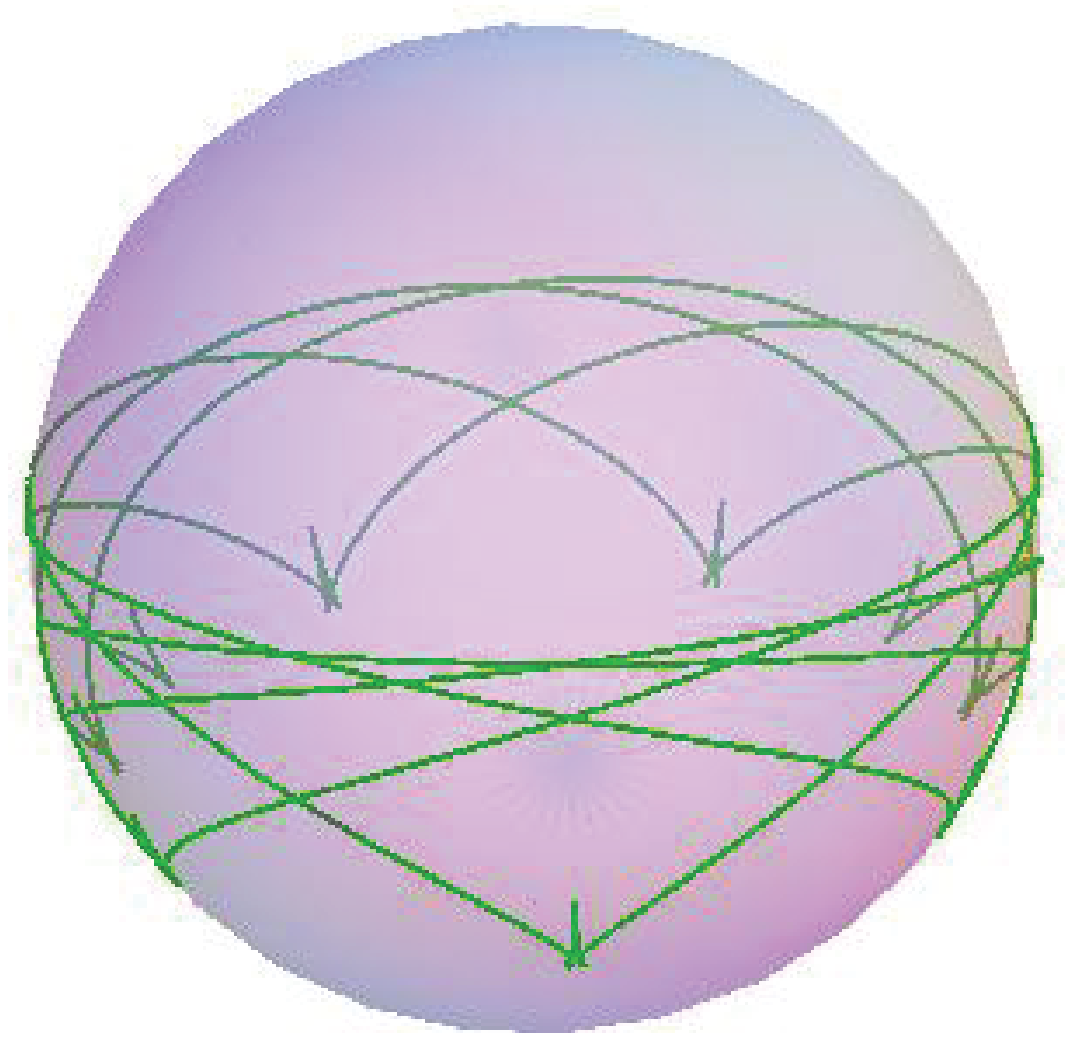}
 }
 \subfigure[]{
   \includegraphics[scale =0.31] {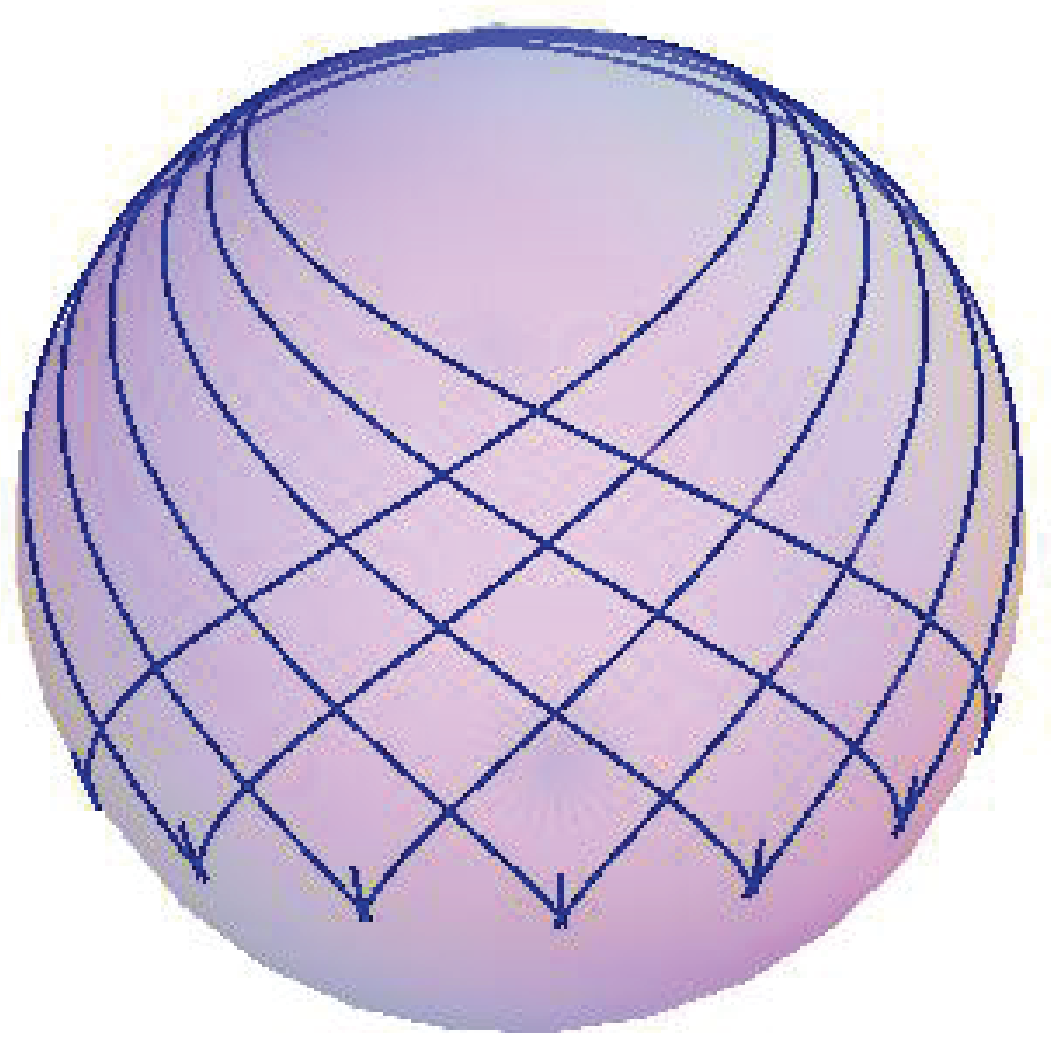}
 }
\label{cases1.5}
\caption{For $A=1$, $B=0$ and $a=2$, the spherical slant helix $\alpha$ (a), For $A=1$, $B=0$ and $a=3$, the spherical slant helix $\alpha$ (b) and For $A=1$, $B=0$ and $a=4$, the spherical slant helix $\alpha$ (c).}
\end{figure}

\begin{figure}[h]
\begin{center}
\includegraphics[width=2.3in]{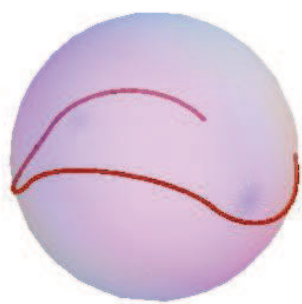}
\caption{For $a=1$, $A=1$, $B=0$, spherical slant helix $Y$.}
\label{1}
\end{center}
\end{figure}

\begin{figure}[ht]
\centering
\subfigure[]{
   \includegraphics[scale =0.35] {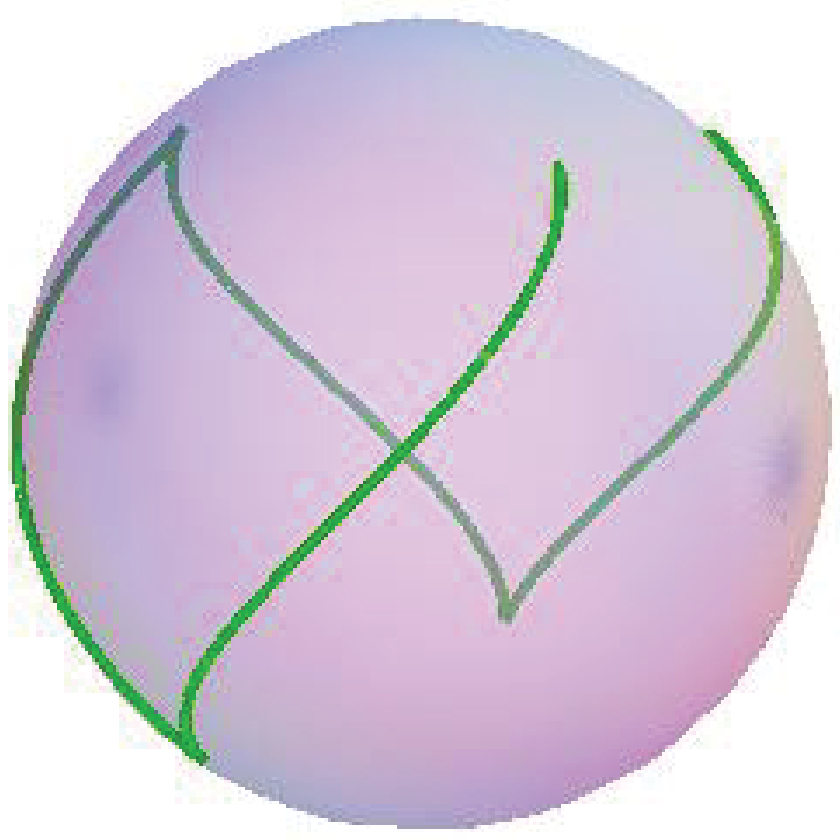}
 }
 \subfigure[]{
   \includegraphics[scale =0.35] {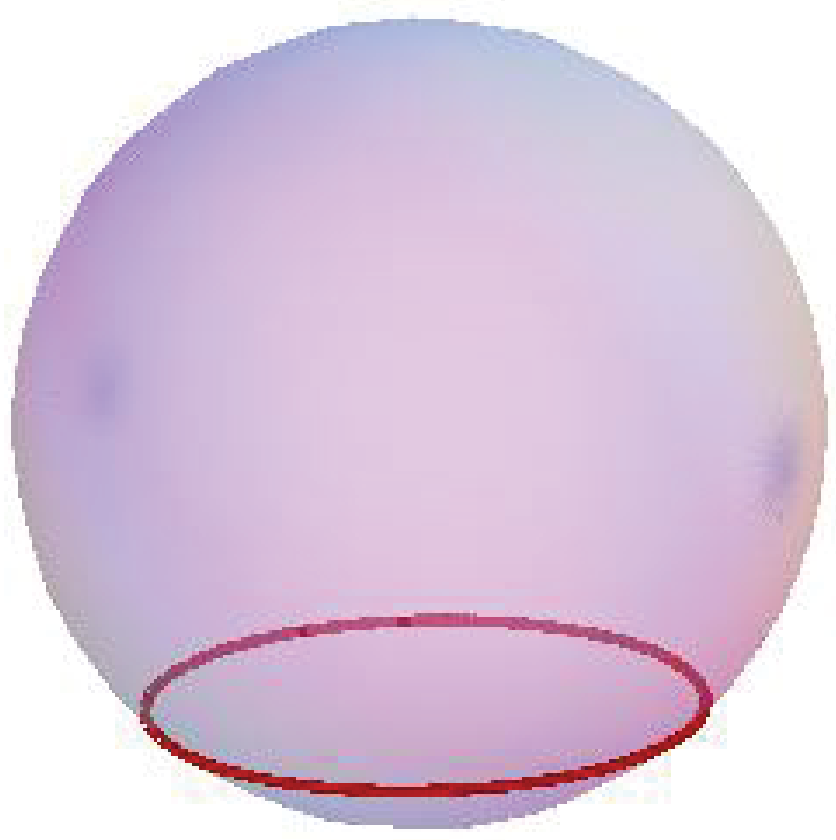}
 }
 \subfigure[]{
   \includegraphics[scale =0.35] {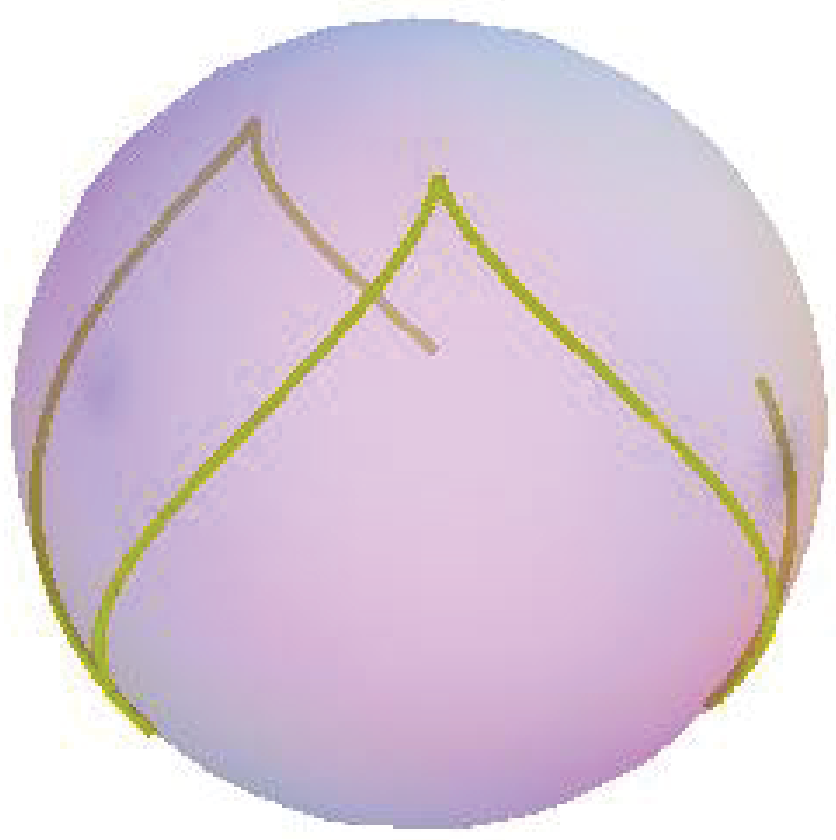}
 }
\label{cases1.5}
\caption{For $a=1$, $A=1$, $B=0$, tangent indicatrix of the spherical slant helix $\alpha$ (a), normal indicatrix of the spherical slant helix $\alpha$ (b) and binormal indicatrix of the spherical slant helix $\alpha$ (c).}
\end{figure}

\end{document}